\documentclass[runningheads]{llncs}
\usepackage[T1]{fontenc}
\usepackage{graphicx}
\usepackage{amsmath, amsfonts, amssymb, bm, stmaryrd, enumerate, mathtools}
\usepackage{enumitem}
\usepackage{comment}
\usepackage{tikz}
\usepackage{tikz-cd}

\newcommand{\C}{{\mbox{${\mathcal C}$}}}
\DeclareMathOperator{\kernel}{ker}
\DeclareMathOperator{\image}{im}

\newcommand{\set}[2]{\{#1\;|\;#2\}}

\newcommand{\downset}{\mathop{\downarrow}\!}

\newcommand{\auv}[1]{``#1''}

\newcommand{\conjun}{\mathrel{\wedge}}
\newcommand{\disjun}{\mathrel{\vee}}

\newcommand{\adj}{^\star}

\newcommand{\after}{\mathrel{\circ}}

\newcommand{\Cat}[1]{\ensuremath{\mathbf{#1}}}

\newcommand{\op}{\ensuremath{^{\mathrm{op}}}}
\newcommand{\idmap}[1][]{\ensuremath{\mathrm{id}_{#1}}}

\newcommand{\nul}{\ensuremath{\underline{0}}}
\newcommand{\sai}[1]{[\,#1\,]}
\newcommand{\EndoHom}[1]{\mathcal{E}{\kern-.5ex}\textit{n}{\kern-.2ex}\textit{d}{\kern-.2ex}\textit{o}(#1)}
\newcommand{\ex}[2]{\exists_{#1}.\,#2}

\begin{document}
\title{Foulis quantales and complete orthomodular lattices}
%
%
\author{Michal Botur\inst{1}\orcidID{0000-0002-2362-5113} \and
Jan Paseka\inst{2}\orcidID{0000-0001-6658-6647} \and
Richard Smolka\inst{2}\orcidID{0009-0009-3809-6232}}
\authorrunning{M. Botur et al.}
%
\institute{Department of Algebra and Geometry, Faculty of Science,
Palack\'y University Olomouc, 17.~listopadu~12, 771 46 Olomouc, Czech Republic\\
\email{michal.botur@upol.cz}\\
\and
Department of Mathematics and Statistics, Faculty of Science, Masaryk University, 
Kotl\'a\v rsk\'a 2, 611~37~Brno, Czech Republic\\
\email{
paseka@math.muni.cz}}
\maketitle              
\begin{abstract}
Our approach establishes a natural correspondence between complete orthomodular lattices and certain types of quantales.

Firstly, given a complete orthomodular lattice X, we associate with it a Foulis quantale $\Cat{Lin}(X)$ consisting of its endomorphisms. This allows us to view X as a left module over $\Cat{Lin}(X)$, thereby introducing a novel fuzzy-theoretic perspective to the study of complete orthomodular lattices.

Conversely, for any Foulis quantale $Q$, we associate a complete orthomodular lattice $\sai{Q}$ that naturally forms a left $Q$-module. Furthermore, there exists a canonical homomorphism of Foulis quantales from $Q$ to $\Cat{Lin}(\sai{Q})$.


\end{abstract}
\section{Introduction}\label{sect:intro}

Fuzzy logic, a framework for reasoning with uncertainty, is fundamentally connected to residuation, a concept central to various branches of mathematics and computer science. Notable examples include Hájek's Basic Logic (BL) \cite{hajek1998} and Esteva and Godo's Monoidal T-norm Based Logic (MTL) \cite{esteva_godo2001}, which arise from continuous and left-continuous t-norms respectively. The essential role of residuated structures in characterizing these and other fuzzy logics has driven extensive research in this area.

Quantales \cite{mulvey}  are essential structures in theoretical computer science and mathematical physics. They underpin significant frameworks like process algebras and quantum logics, developed by 
Abramsky and Vickers \cite{abramsky_vickers1993}, as well as the quantum mechanical formalisms of Birkhoff and von Neumann \cite{Birkhoff_Neumann1936}. Quantales naturally model concurrent processes and algebraically represent quantum mechanics. Their semantics can be formally described using quantale modules, leading to extensive study of quantales and their modules for their relevance in quantum logic, concurrency theory, and related fields.

We study Foulis quantales and their modules as quantale-like structures for complete orthomodular lattices. This approach introduces a fuzzy-theoretic perspective, as every quantale module naturally carries a fuzzy order \cite{Sol09}.

This paper is structured as follows. Section 2 provides an overview of fundamental algebraic concepts, including orthomodular lattices, dagger categories, quantales, and quantaloids. Drawing inspiration from research on orthomodular lattices \cite{Jac,BLP} and complete orthomodular lattices \cite{BPS}, we investigate the essential properties of the involutive quantale of endomorphisms of a complete orthomodular lattice. This investigation lays the groundwork for a fuzzy-logic-inspired approach to quantum logic.

Section 3 delves into the core theme of the paper: Foulis quantales and their modules. We explore their key characteristics and examine their applications within the context of complete orthomodular lattices. 

This work assumes familiarity with the foundational concepts and results pertaining to quantales, dagger categories and orthomodular lattices. Readers seeking further information on these topics are encouraged to consult \cite{KIR,Rosenthal1996}, \cite{HeJa,Jac}, and \cite{Kalmbach83}.

\section{Complete orthomodular lattices, dagger categories and quantales}\label{sect:prelim}
\subsection{Complete orthomodular lattices} 

The concept of ortholattices provides a broader mathematical framework that extends beyond traditional Boolean algebras. These structures are distinguished by their orthocomplement operation - a sophisticated counterpart to Boolean negation. When we add the property of orthomodularity, we arrive at orthomodular ortholattices, which form a special class with unique characteristics. This refinement proves especially valuable in quantum logic and related fields, where the additional structure captures important mathematical and logical relationships that Boolean algebras cannot express.

\begin{definition}\label{OMLatDef}\label{def:linear_map}{\em  \cite{Jac}
A meet semi-lattice $(X,\conjun, 1)$ is called an {\em ortholattice} if it
comes equipped with a function $(-)^{\perp}\colon X \to X$ satisfying:
\begin{itemize}
   \item $x^{\perp\perp} = x$;
   \item $x \leq y$ implies $y^\perp \leq x^\perp$;
   \item $x \conjun x^\perp = 1^\perp$.
\end{itemize}

\noindent One can then define a bottom element as $0 = 1 \conjun
1^{\perp} = 1^\perp$ and join by $x\disjun y = (x^{\perp}\conjun
y^{\perp})^{\perp}$, satisfying $x\disjun x^{\perp} = 1$.

We write $x\perp y$ if and only if $x\leq y^{\perp}$. 

Such an ortholattice is called {\em orthomodular lattice} if it satisfies (one of)
the three equivalent conditions:
\begin{itemize}
\item $x \leq y$ implies $y = x \disjun (x^\perp \conjun y)$;

\item $x \leq y$ implies $x = y \conjun (y^\perp \disjun x)$;

\item $x \leq y$ and $x^{\perp} \conjun y = 0$ implies $x=y$.
\end{itemize}
An orthomodular lattice $X$  is called a {\em complete orthomodular lattice} if it is also a complete lattice.
}
\end{definition}

We introduce a categorical framework for complete orthomodular lattices by constructing a category $\mathbf{SupOMLatLin}$. The objects are complete orthomodular lattices, and morphisms are linear maps between them (following the treatment in \cite{BLP,BPS}).

\begin{definition}\label{def:SupOMLatLin}{\em \cite{BPS}
The category \Cat{SupOMLatLin} has complete orthomodular
lattices as objects.
A morphism $f \colon X\rightarrow Y$ in \Cat{SupOMLatLin} is a 
function $f \colon X\rightarrow Y$ between the underlying sets such that 
there is a function $h \colon Y \to X$ and, 
for any $x \in X$ and $y \in Y$,
\[ f(x) \perp y \text{ if and only if } x \perp h(y). \]
We say that $h$ is an {\it adjoint} of a {\em linear map} $f$. 
It is clear that adjointness is a symmetric property: if a map $f$ possesses an adjoint $h$, then $f$ is also an adjoint of $h$. 
We denote $\Cat{Lin}(X,Y)$ the set of all linear maps from $X$ to $Y$.
If $X=Y$ we put $\Cat{Lin}(X)=\Cat{Lin}(X,X)$.

Moreover, a map $f \colon X \to X$ is called {\it self-adjoint} if $f$ is an adjoint of itself.

The identity morphism on $X$ is the self-adjoint identity map $\idmap[X] \colon X\rightarrow X$. Composition of $\smash{X \stackrel{f}{\rightarrow} Y
  \stackrel{g}{\rightarrow} Z}$ is given by usual composition of maps.}
\end{definition}

Our guiding example is the following construction. Let $\mathcal{H}$ be a Hilbert space and consider $\mathcal{C}(\mathcal{H})$, the set of all closed subspaces of $\mathcal{H}$. When equipped with the operations $\wedge$ (intersection) and $\perp$ (orthogonal complementation), $\mathcal{C}(\mathcal{H})$ forms a complete orthomodular lattice. Moreover, for Hilbert spaces $\mathcal{H}_1$ and $\mathcal{H}_2$, any bounded linear operator $T: \mathcal{H}_1 \rightarrow \mathcal{H}_2$ induces a linear map $\Phi_T: \mathcal{C}(\mathcal{H}_1) \rightarrow \mathcal{C}(\mathcal{H}_2)$ between the corresponding lattices of closed subspaces. The adjoint map $\Phi_T^*$ is naturally determined by $T^*$, the adjoint of the original operator.

\begin{lemma} \label{lem:lattice-adjoint}{\rm \cite[Lemma 2.6]{BPS}}
Let $f: X \rightarrow Y$ be a map between complete orthomodular lattices. The following  three key properties of $f$ are equivalent:
\begin{enumerate}
    \item $f$ possesses a right order-adjoint;
    \item $f$ admits an adjoint in the sense of Definition~\ref{def:linear_map};
    \item $f$ preserves arbitrary joins (i.e., is join-complete).
\end{enumerate}
\end{lemma}
This equivalence provides multiple perspectives for understanding linear maps in the context of complete orthomodular lattices.

The Sasaki projection, named after Shôichirô Sasaki, is a crucial concept in orthomodular lattice theory. It defines a unique projection operation that reflects the non-classical, quantum-like behavior of these structures, offering insights into relationships between elements beyond the scope of classical Boolean logic.

\begin{definition}\label{def:Sasaki projection}
		Let $X$ be an orthomodular lattice. Then the map $\pi_a:X\to X$, $y\mapsto a\wedge(a^\perp\vee y)$ is called the \emph{Sasaki projection} to $a\in X$.
	\end{definition}

    The following properties of Sasaki projections are established in  \cite{LiVe} and are crucial for our discussion:

    \begin{lemma}\label{lem:Sasaki projection facts}
		Let $X$ be an orthomodular lattice, and let $a\in X$. Then for each $y,z\in L$ we have
		\begin{itemize}
			\item[(a)] $y\leq a$ if and only if $\pi_a(y)=y$;
			\item[(b)] $\pi_a(\pi_a(y^\perp)^\perp))\leq y$;
			\item[(c)] $\pi_a(y)=0$ if and only if $y\leq a^\perp$;
			\item[(d)] $\pi_a(y)\perp z$ if and only if $y\perp \pi_a(z)$.
		\end{itemize}
	\end{lemma}

    The definition that follows brings to mind the kernel and range, two key ideas in the study of morphisms between complete orthomodular lattices. One important difference between them is that the range only produces a complete lattice, but a kernel always forms a complete orthomodular lattice.

    \begin{definition}\label{def:kernel}
    Let $f \colon X \to Y$ be a morphism of complete orthomodular lattices. 
    We define the {\it kernel} and the {\it range} of $f$, respectively, by
\begin{align*}
\kernel f \;=\; & \{ x \in X \colon f(x) = 0 \}, \\
\image f \;=\; & \{ f(x) \colon x \in X \}.
\end{align*}
\end{definition}

\subsection{Dagger categories}

Dagger categories, crucial in categorical quantum mechanics, are categories with an involutive functor (the dagger) defining adjoints for morphisms. This structure enables the representation and analysis of quantum processes and reversible computations.

We observe that this idea has been present in the literature since the 1960s and was usually taken into consideration in a particular context. With the publication of Abramsky and Coecke's work \cite{AbCo}, it became part of the mainstream discourse regarding the fundamentals of quantum mechanics. P. Selinger is credited with coining the term \auv{dagger category} \cite{Sel}.

\begin{definition}
\label{DagcatDef}\rm 
A {\it dagger} on a category \C\ is a functor ${}^\star \colon \C\op \to \C$ that is involutive and the identity on objects. A category equipped with a dagger is called a {\it dagger category}.

Let \C\ be a dagger category. A morphism $f \colon A \to B$ is called a {\it dagger monomorphism} if $f^{\star} \circ f={\idmap}_A$, and $f$ is called a {\it dagger isomorphism} if 
$f^{\star} \circ f = {\idmap}_A$ and $f \circ f^\star = {\idmap}_B$. A {\it dagger automorphism} is a dagger isomorphism $f \colon A \to A$.
\end{definition}

Limits and colimits are dual in dagger categories. Applying the dagger ${}\adj$ to a limit cone yields a colimit cone, and conversely.

The category of complete orthomodular lattices with linear maps is shown to constitute a dagger category by the following theorem. 

\begin{theorem}\label{OMLisdagger}{\rm\cite[Theorem 2.7]{BPS}}
    \Cat{SupOMLatLin} is  a dagger category.
\end{theorem}

Two significant findings that are worth mentioning are as follows.

  \begin{corollary}\label{Sasself}{\rm \cite[Corollary 3.5, Corollary 4.6]{BPS}}
     Let $X$ be a complete orthomodular lattice, and let $a\in X$. Then 
     $\pi_a$ is self-adjoint, idempotent and 
     $\image \pi_a=\downset a= \downset\pi_a(1)$. Moreover, we have a factorization 
     \begin{equation}
  \label{OMLImageEqn}
\begin{tikzcd}
& \downset a\arrow[r, hook, "\pi_{a}|_{\downset a}"] & X \\
& X \arrow[u, "\pi_{a}|^{\downset a}", dashed] \arrow[ur, "\pi_a"'] & 
\end{tikzcd}
\end{equation}
such that $\pi_{a}|^{\downset a}$ is dagger epi, $\pi_{a}|_{\downset a}$ is dagger mono, 
$\pi_{a}|_{\downset a}\circ \pi_{a}|^{\downset a}=\pi_{a}$, $\pi_{a}|^{\downset a}=\left(\pi_{a}|_{\downset a}\right)^{*}$, and 
$\pi_{a}|^{\downset a}\circ \pi_{a}|_{\downset a}=\idmap[{\downset a}]$.
 \end{corollary}

 \begin{lemma} \label{KernelLemma}{\rm \cite[Corollary 3.9]{BPS}}
Let $f \colon X \to Y$ be a morphism of complete orthomodular lattices. Then 
$\kernel f=\downset f^{*}(1)^{\perp}$ is a complete orthomodular lattice.
\end{lemma}

The category \Cat{SupOMLatLin} possesses a zero object $\nul$ (see \cite{BPS}), characterized by the existence of unique morphisms to and from any object. This zero object is the one-element orthomodular lattice $\{0\}$. For any complete orthomodular lattice X, the unique morphism $\nul \to X$ maps 0 to 0.

Identifying $\nul$ with the principal downset $\downset 0$, this morphism is a dagger mono\-morphism with adjoint $f^*: X \to \nul$ given by 
$f^*(x)=\pi_0(x)=0$. We denote by $0_{X,Y}: X \to \nul \to Y$ the unique morphism factoring through the zero object for any objects $X$ and $Y$.

The following  definition introduces kernels and their dagger counterparts, laying the foundation for studying algebraic structures within the framework of category theory.

\begin{definition}\label{weakkernel}{\em \cite{HeJa}
\begin{enumerate}
\item[]
\item For a morphism $f \colon A \to B$ 
in arbitrary category with zero morphisms, we say that a morphism  
$k\colon  K \to A$ is a {\em  kernel} of $f$ if 
$fk=0_{K,B}$, and if $m\colon M\to A$ satisfies $fm=0_{M,B}$  then  there is a unique morphism 
$u \colon M \to K$ such that $ku = m$. 

We sometimes write $\kernel f$ for $k$ or $K$.

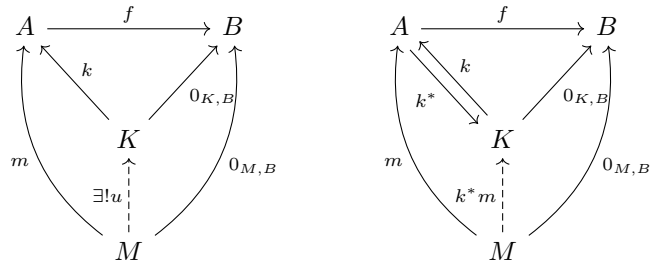
\[
\begin{array}{@{}c c c}
\begin{tikzcd}[row sep=1cm, column sep=.75cm, ampersand replacement=\&]
	A\arrow[rr, "f"] \& \& B\\
	\& K\arrow[ul, "k"']\arrow[ur, "0_{K,B}"']\\
	\& M\arrow[uul, bend left, "m"]\arrow[u,dashed,"\exists!u"]\arrow[uur, bend right, "0_{M,B}"']
\end{tikzcd}&\quad\quad\quad &
\begin{tikzcd}[row sep=1cm, column sep=.75cm, ampersand replacement=\&]
	A\arrow[rr, "f"] \arrow[dr, shift right, swap, "k^{*}"]\& \& B\\
	\& K\arrow[ul,shift right, "k"']\arrow[ur, "0_{K,B}"']\\
	\& M\arrow[uul, bend left, "m"]\arrow[u,dashed,"k^{*}m"]\arrow[uur, bend right, "0_{M,B}"']
\end{tikzcd}
\end{array}
\]

 \item   For a morphism $f \colon A \to B$ 
in arbitrary dagger category with zero morphisms, we say that a morphism  
$k\colon  K \to A$ is a {\em weak dagger kernel} of $f$ if 
$fk=0_{K,B}$, and if $m\colon M\to A$ satisfies $fm=0_{M,B}$  then $kk^{*}m=m$.

A {\em weak dagger kernel category} is a dagger category with zero morphisms where every morphism has a weak dagger kernel.

\item A {\em dagger kernel category} is a dagger category with a zero object, hence zero morphisms, where each morphism $f$ has a weak dagger kernel $k$ (called {\em dagger kernel}) 
that additionally satisfies $k^{*}k=1_K$.

\end{enumerate}}
\end{definition} 

Every dagger kernel is a kernel, and in $\Cat{SupOMLatLin}$, the converse holds: all kernels are dagger kernels. Theorem \ref{OMLatLinDagKerCatThm} thus shows that $\Cat{SupOMLatLin}$ forms a dagger kernel category.

\begin{theorem}{\rm \cite[Theorem 3.12]{BPS}}
\label{OMLatLinDagKerCatThm}
  The category $\Cat{SupOMLatLin}$ is a dagger kernel category. The
  dagger kernel of a morphism $f \colon X \to Y$ is 
  $\kernel(f)=\pi_{k}|_{\downset k} \colon\!  \downset k \to X$, where $k = f^*(1)^{\perp} \in X$, like in
  Lemma~\ref{KernelLemma}. Moreover, $f\circ \kernel(f)=0_{\downset k, Y}$, $\pi_{k}=\kernel(f)\circ \kernel(f)^{*}$ and 
  $\idmap[{\downset k}]=\kernel(f)^{*}\circ \kernel(f)$.
\end{theorem}

\subsection{Quantales and quantaloids}

While complete orthomodular lattices give us a snapshot of a quantum system's possible states, focusing on testable properties, quantales take a different viewpoint. They provide a dynamic perspective, allowing us to reason about the evolution of the system and the structure of quantum actions that cause these changes.

Mulvey (\cite{mulvey}) coined the term \auv{quantales} as a \auv{quantization} of the term \auv{location} during the Oberwolfach Category Meeting in the early 1980s.  Quantales, which were first inspired by studies in topology and functional analysis, provide a strong foundation for simulating complex systems with non-classical logics, especially in domains such as computer science, abstract algebra, and quantum mechanics. The realization that quantales, like Boolean algebras for classical propositional logic, give semantics for propositional linear logic was a crucial advancement in quantales theory. Naturally, quantales appear as subgroups, lattices of ideals, or other appropriate algebraic substructures.

\begin{definition}\label{qsemiq}{\rm \cite{KIR}
\begin{enumerate}
\item A {\em quantale\/} is a complete lattice $Q$ with an associative
binary multiplication satisfying
$$
x\cdot\bigsqcup\limits_{i\in I}
x_i=\bigsqcup\limits_{i\in I}(x\cdot
x_i)\ \ \hbox{and}\ \ \left(\bigsqcup\limits_{i\in I}x_i\right)\cdot
x=\bigsqcup\limits_{i\in I}(x_i\cdot x)
$$
for all $x,\,x_i\in Q,\,i\in I$ ($I$ is a set). Here 
$\bigsqcup\limits_{i\in I} x_i$ denotes the join of the set 
$\{x_i\colon {i\in I}\}$.
A quantale $Q$ is said to be {\em unital\/}  if
there is an element $e\in Q$ called {\em unit} such that
$$
e\cdot a = a = a\cdot e
$$
\noindent
for all $a\in Q$. The element 
$0=\bigsqcup \emptyset$ is a {\em zero
element} of $Q$: $0 \cdot s = 0 = s\cdot 0$.

We denote by $\sqsubseteq$ the order relation on $Q$.

\item By an {\em involutive  quantale} will be meant
a  quantale $Q$ together with a semigroup
 involution $^{*}$ satisfying
$$
(\bigsqcup a_{i})^{*}=\bigsqcup a_{i}^{*}
$$
\noindent for all $a_{i}\in Q$. In the event that $Q$
is also unital, then necessarily $e$ is selfadjoint, i.e.,
$$
e=e^{*}.
$$
 
\end{enumerate}
We also define $s\leq t$ if 
and only if $s=t\cdot s$, and $s\perp t$ if 
and only if $0=s^{*}\cdot t$ for all $s, t\in Q$.}
\end{definition}

Quantaloids \cite{Rosenthal1996} are a versatile and powerful generalization of quantales, enriched in sup-lattices, and are used extensively in category theory and logic. They provide a robust framework for studying enriched categories, automata theory, and other mathematical structures, highlighting their importance in both theoretical and applied mathematics.

\begin{definition}\label{def:quantaloid}{\em 
\begin{enumerate}
\item A {\em quantaloid} \Cat{Q}
 is a category enriched over the category \Cat{Sup} of complete lattices with supremum preserving maps.

 This means that for any objects $A$ and $B$ in the quantaloid, 
 the hom-object $Hom(A,B)$ is not merely a set but a complete lattice, such that the composition of morphisms preserves all joins:
 \begin{align*}
\left(\bigvee_{i\in I} f_i\right)\after 
\left(\bigvee_{j\in J} g_j\right)=
\left(\bigvee_{i\in I, j\in J} f_i\after g_j\right)
 \end{align*}
 
 \item A quantaloid $\Cat{Q}$ with involution, i.e., a dagger category 
  for which 
\begin{align*}
\left(\bigvee_i p_i\right)^* &= \bigvee_i p_i^* 
\end{align*}
for all morphisms $p_i \in \Cat{Q}$, will be called {\em involutive}.
\end{enumerate}
}
\end{definition}

 \begin{example} (see \cite[page 355]{mulvey_pelletier} and 
 \cite[Example 2.4]{Gylys1999Involutive})
 Let \Cat{SupOLatLin} be the category of complete orthocomplemented  lattices with join-preserving maps as morphisms. The composition in \Cat{SupOLatLin} is given by the standard composition of mappings, and the identity morphisms serve as units. One can naturally view \Cat{SupOLatLin} as an involutive quantaloid, where the join is defined by the pointwise ordering of mappings, and the involution ${}^{\dagger}$ on \Cat{SupOLatLin} is given by
\begin{align*}
{f^{\dagger}}(t) = (\bigvee\{s\in X\mid f(s)\leq t^{\perp}\})^{\perp} 
\end{align*}
for each morphism $f\colon X \to Y$ in \Cat{SupOLatLin} and every element $t \in Y$.
\end{example}

\begin{lemma}\label{dagjeadj}
Let $f: X \rightarrow Y$ be a linear map between complete orthomodular lattices. Then ${f}^{\dagger}=f\adj$. 
\end{lemma}
\begin{proof} Let $x\in X$ and $y\in Y$. We compute:
\begin{align*}
&x\perp {f^{\dagger}}(y)
   \text{ if and only if } 
   x\leq \bigvee\{s\in X\mid f(s)\leq y^{\perp}\}
    \text{ if and only if } \\
   &x\leq \bigvee\{s\in X\mid f(s)\perp y\}
    \text{ if and only if } 
   x\leq \bigvee\{s\in X\mid s\perp f\adj(y)\}\\
    &\text{if and only if } 
   x\perp f\adj(y).
\end{align*}
We conclude that ${f}^{\dagger}=f\adj$. 
\end{proof}

Given that the category \Cat{SupOLatLin}  of complete orthocomplemented  lattices satisfies the definition of an involutive quantaloid, we can directly derive the following theorem based on Theorem \ref{OMLisdagger} and Lemma \ref{dagjeadj}.


\begin{theorem}
\label{OMLatLinsemiq}  \Cat{SupOMLatLin} is an involutive quantaloid that is a full dagger subcategory of \Cat{SupOLatLin}.
\end{theorem}

The intrinsic structure of sets of linear maps between complete orthomodular lattices is the focus of our current discussion.
We derive the following as a corollary of Theorem \ref{OMLatLinsemiq}.

\begin{corollary}\label{Linisquant}
Let $X$ and $Y$ be complete orthomodular lattices. Then 
     \begin{enumerate}[label={\rm ({\roman*})}]
     \item $\Cat{Lin}(X,Y)$ is a complete lattice,
         \item $\Cat{Lin}(X)$ is a unital involutive quantale.
     \end{enumerate}
\end{corollary}

\section{Foulis quantales and their modules}

In the 1960s, David Foulis introduced a novel mathematical structure known as a ``Baer *-semigroup'', later termed a ``Foulis semigroup''. Properties of the multiplicative semigroup of bounded operators on a Hilbert space served as the basis for Foulis' initial set of characteristics.

This structure provides a framework for modeling various aspects of quantum mechanics. While it has been referenced under different names throughout the literature, its fundamental properties have remained unchanged. For a concise overview, we refer to Chapter 5, Section 18 of Kalmbach's book~\cite{Kalmbach83}. The Foulis quantales we introduce here can be characterized precisely as unital involutive quantales that additionally exhibit the structural properties of Foulis semigroups.

\begin{definition}{\em 
    A {\em Foulis quantale} is a unital involutive quantale $Q$ 
together with an endomap $\sai{-} \colon Q\rightarrow Q$ satisfying:
\begin{enumerate}[label=({\alph*})]
\item $\sai{s}$ is a self-adjoint idempotent, {i.e.},~satisfies
  $\sai{s} \cdot \sai{s} = \sai{s} = \sai{s}^{*}$;

\item $0\, {=} \, \sai{e}$;

\item $s\cdot x = 0$ iff $\ex{y}{x = \sai{s}\cdot y}$.

\end{enumerate}
For an arbitrary $t\in Q$ put $t^{\perp}
  \,\smash{\stackrel{\textrm{def}}{=}}\, \sai{t^{*}} \in \sai{Q}$.
  Hence from~(a) we get equations $t^{\perp} \cdot t^{\perp} =
  t^{\perp} = (t^{\perp})^{*}$. %
  We will call elements of $\sai{Q}$ {\em Sasaki projections}.

  A homomorphism of Foulis quantales is a 
map $h\colon Q_1 \to Q_2$ between Foulis quantales 
that preserves arbitrary joins, finite multiplication, unit, involution, 
and ${}^\perp$. In particular, $h$ maps Sasaki projections to Sasaki projections.}
\end{definition}

\begin{remark}\label{remFouldef}\rm 
    A quantale is a Foulis quantale if and only if it is a unital involutive quantale $Q$ equipped with an endomap ${-}^{\perp} \colon Q\rightarrow Q$ satisfying the following conditions:
\begin{enumerate}[label=({\arabic*})]
\item ${s}^{\perp}$ is a self-adjoint idempotent, i.e., ${s}^{\perp} \cdot {s}^{\perp} = {s}^{\perp} = \big({s}^{\perp}\big)^{*}$;
\item $0 = {e}^{\perp}$;
\item $s\perp x = 0$ if and only if there exists $y$ such that $x = {s}^{\perp}\cdot y$.
\end{enumerate}
\end{remark}

When considering a complete orthomodular lattice, its associated endomorphism quantale exhibits the specific algebraic properties that characterize a Foulis quantale. The fact that $\Cat{Lin}(X)$ forms a Foulis semigroup for any orthomodular lattice $X$ is a classical result from 
\cite{Foul}, later elaborated via Galois connections in \cite[Chapter~5, \S\S18]{Kalmbach83}, dating back six decades.

\begin{proposition}\label{prop:inv}
     Let $X$ be a complete orthomodular lattice. Then 
     $\Cat{Lin}(X)$ is a Foulis quantale.
 \end{proposition}
 \begin{proof} We know from Corollary \ref{Linisquant} that 
 $\Cat{Lin}(X)$ is a unital involutive quantale.
Let us define the endomap 
$\sai{-} \colon \Cat{Lin}(X)\rightarrow \Cat{Lin}(X)$ by 
 $\sai{s} =\pi_{s^*(1)^{\perp}}$ for all $s\in \Cat{Lin}(X)$. 
 Evidently,  $\sai{s}$  is a self-adjoint idempotent. 
 We compute:
 \begin{align*}
\sai{\idmap[{X}]}=\pi_{\idmap[X]^*(1)^{\perp}}=%
\pi_{1^{\perp}}=\pi_{0}=0_{X,X}.
 \end{align*}
\noindent Suppose that $s, t\in \Cat{Lin}(X)$. Let $t=\sai{s}\after r$. From Theorem \ref{OMLatLinDagKerCatThm} 
we conclude
\begin{align*}s \after t
&=s \after \pi_{s^*(1)^{\perp}} \after r=
s \after \ker(s) \after \ker(s)^{*} \after r\\
&=
0_{\downset {s^{*}(1)}^{\perp}, X} \after \ker(s)^{*} \after r 
=0_{X,X}=0_{\Cat{Lin}(X)}.
\end{align*}

\noindent Conversely, if $s\after t = 0_{\Cat{Lin}(X)}=0_{X,X}$, then there is a linear map
$f\colon X\to {\downset {s^{*}(1)}^{\perp}} $ such 
that $\ker(s) \after f = t$ and 
$0_{\downset s^*(1)^{\perp},X}\after f=0_{X,X}$. Hence $t$ satisfies:
$$\sai{s} \after t
= \pi_{s^*(1)^{\perp}} \after t=
\ker(s) \after \ker(s)^{*} \after \ker(s) \after f
=
\ker(s) \after f
=
t.
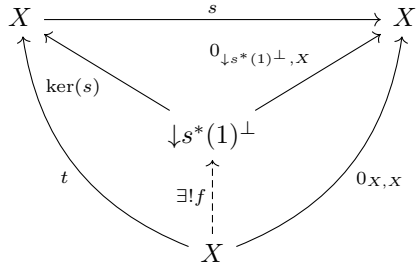
$$

\[
\begin{tikzcd}[row sep=1cm, column sep=.75cm, ampersand replacement=\&]
	X\arrow[rrrr, "s"] \& \& \& \&X\\
	\&\& \downset s^*(1)^{\perp}\arrow[ull, "\kernel(s)"]\arrow[urr, "0_{\downset s^*(1)^{\perp},X}"]\&\&\\
	\&\& X\arrow[uull, bend left, "t"]\arrow[u,dashed,"\exists!f"]\arrow[uurr, bend right, "0_{X,X}"']
\end{tikzcd}
\]
\end{proof}

\begin{theorem}
\label{FoulisOMKerLem}
Let $Q$ be a Foulis quantale. Then, 
for all $t, r\in Q$ and $k\in \sai{Q}$, 
$$\begin{array}{rcl}
r^{*}\cdot t=0
& \Longleftrightarrow &
t=[r^{*}]\cdot t,
\end{array}\eqno{(*)\phantom{**}}$$
$$\begin{array}{rclcl}
t\leq  r
& \Longrightarrow &
r^{\perp}\leq t^{\perp} &\text{and}&
k^{\perp\perp}=k,
\end{array}\eqno{(**)\phantom{*}}$$
$$\begin{array}{rcl}
t\leq  r^{\perp}
& \Longleftrightarrow &
r\leq t^{\perp}. 
\end{array}\eqno{(***)}$$

\noindent{}and the subset
$$\begin{array}{rcccl}
\sai{Q}
& = &
\set{\sai{t}}{t\in Q} 
& \subseteq & 
Q,
\end{array}$$

\noindent is a complete orthomodular lattice with the following structure.
$$\begin{array}{lrcl}
\mbox{Order} & k_{1}\leq k_{2} & \Leftrightarrow & k_{1} = k_{2}\cdot k_{1} \\
\mbox{Top} & 1 & = &  
   \sai{0} \\
\mbox{Orthocomplement\quad\quad} & k^{\perp} & = & \sai{k} \\
\mbox{Meet} & k_{1} \conjun k_{2} & = &  
   \big(k_{1} \cdot \sai{\sai{k_{2}}\cdot k_{1}}\big)^{\perp\perp}\\
\mbox{Join} & \bigvee S & = &  \sai{\sai{\bigsqcup S}}.
\end{array}$$
\end{theorem}
\begin{proof} 
A significant portion of this statement, including its proof for Foulis semigroups, aligns with the established findings presented in \cite[Lemma 4.6.]{Jac}. To complete the demonstration, it is necessary  only to establish that
\begin{equation*}
\sai{\sai{\bigsqcup S}} = \bigvee S
\end{equation*}
where $S$ is a subset of $\sai{Q}$. As this portion of the proof is identical to the argument presented in \cite[Theorem 21]{BLP} for a finite subset, we refer the reader there for the complete details.
\end{proof}

Quantale modules formalize the interaction between quantum operations and quantum states. Specifically, elements of a quantale ${Q}$ represent quantum operations, while points in a module ${A}$ correspond to the possible states of a quantum system. The action of the quantale on its module captures how quantum operations transform the system state: each operation $q \in {Q}$ induces a state transition $q \cdot m \mapsto m'$, where $m, m' \in {A}$.

\begin{definition}\label{keymodule}{\em  Given a quantale $Q$, 
a {\em  left $Q$-module} 
	is a complete lattice $A$  and a map $\bullet\colon Q\times A\longrightarrow A$ satisfying:
	\begin{itemize}
		\item[](A1) $s \bullet (\bigvee B)=\bigvee_{x\in B}(s \bullet x)$ 
		for every  $B\subseteq A$ and  $s \in Q$.
		\item[](A2) $(\bigsqcup T)\bullet a=\bigvee_{t\in T}(t\bullet a)$ 
		for every $T\subseteq Q$ and  $a \in A$.
		\item[](A3) $u\bullet(v\bullet a)=(u\cdot v)\bullet a$ for every $u,v \in Q$ and every $a\in A$.
		\item[](A4) $e \bullet a=a$ for all $a\in A$ (unitality).\\
	\end{itemize}	}
\end{definition}

The definition of right $Q$-modules follows analogously. It is readily apparent that every complete lattice $A$ forms a right ${\mathbf 2}$-module. Here, ${\mathbf 2}$ represents a two-element chain equipped with a meet operation as its multiplication and the identity map as its involution.

The following statement says that  a complete orthomodular lattice $X$  can be acted upon from the left by its linear transformations and from the right by a specific two-element structure, giving it two different but compatible ways of being transformed or modified.

\begin{proposition}\label{rqinv}
	Let $X$ be a complete orthomodular lattice. Then 
	$X$ is a left $\Cat{Lin}(X)$-module and also a right $\Cat{2}$-module.
\end{proposition}
\begin{proof} We define the action 
$\bullet\colon \Cat{Lin}(X)\times X\longrightarrow X$ by 
$f\bullet x=f(x)$ for all $f\in \Cat{Lin}(X)$ and all $x\in X$. 
The verification of conditions (A1)-(A4) is transparent.
\end{proof}

The following theorem demonstrates that the complete orthomodular lattice of Sasaki projections in a Foulis quantale carries both a left module structure over the quantale itself and a right module structure over the two-element Boolean algebra. This bimodule structure illuminates the algebraic nature of Sasaki projections.

\begin{theorem} Let $Q$ be a Foulis quantale. Then $\sai{Q}$ is a left $Q$-module   with action $\bullet$ defined as 
$u\bullet k=(u\cdot k)^{\perp\perp}$ for all $u\in Q$ and $k\in \sai{Q}$  and also a right $\Cat{2}$-module.
\end{theorem}
\begin{proof} The proof proceeds analogously to \cite[Theorem 24]{BLP}, with the sole modification that arbitrary subsets are used in place of finite subsets throughout the argument.
\end{proof}

\begin{definition}\label{Sasaki action}
    Let $Q$ be a Foulis quantale and $u\in Q$. 
    Then the map $\sigma_u:\sai{Q}\to \sai{Q}$, $y\mapsto u\bullet  y$ is called the \emph{Sasaki action} to $u\in Q$.
\end{definition}

Evidently, $\sigma_u\in \Cat{Lin}(\sai{Q})$. Moreover, if 
$u\in \sai{Q}$ then $\sigma_u$ is  self-adjoint linear, idempotent 
    and $\image \sigma_u=\downset u$ in $\sai{Q}$ (see \cite[Proposition 26]{BLP}).

The following theorem establishes a canonical correspondence between elements of a Foulis quantale and linear transformations acting on its Sasaki projections, illuminating the structural relationship between these components.

\begin{theorem} Let $Q$ be a Foulis quantale. Then there is a natural 
homorphism $h\colon Q\to \Cat{Lin}(\sai{Q})$ of Foulis quantales 
such that $h(u)=\sigma_u$ for all $u\in Q$.
\end{theorem}
\begin{proof} Since $\sai{S}$ is a left $Q$-module, we obtain that $h$ preserves multiplication, arbitrary joins, and the unit. It remains to show that $h$ preserves involution and ${}^{\perp}$.

Assume that $u\in Q$ and $k,l\in \sai{Q}$. We compute:
\begin{align*}
    &h(u)\adj (k)\leq l^{\perp}  \text{ if and only if } 
    h(u)\adj (k) \perp l  \text{ if and only if } 
    k\perp h(u)(l)\\
    &\text{if and only if } h(u)(l)\leq k^{\perp} 
    \text{ if and only if } u\bullet l \leq k^{\perp} 
    \text{ if and only if } \\
    &k^{\perp}  \cdot u\cdot l =  u\cdot l \text{ if and only if } 
    k \cdot u\cdot l =0 \text{ if and only if } 
    l\adj \cdot u\adj \cdot k\adj =0 \\
    &\text{if and only if }  l\cdot u\adj \cdot k =0 
    \text{ if and only if } l^{\perp}  \cdot u\adj \cdot k =  u\adj\cdot k\\
    &\text{if and only if } u\adj \bullet k\leq l^{\perp} 
     \text{ if and only if } h(u\adj)(k)\leq l^{\perp}.
\end{align*}
Since 
$u^{\perp}\in \sai{Q}$ we  have that 
$h(u^{\perp})=\sigma_{u^{\perp}}$ is the classical Sasaki projection 
$\pi_{u^{\perp}}$ by \cite[Proposition 3.6]{BPS}. From 
Proposition~\ref{prop:inv} we obtain 
\begin{align*}
    h(u)^{\perp}=\pi_{h(u)(e)^{\perp}}=%
    \pi_{(u^{\perp\perp})^{\perp}}=\pi_{u^{\perp}}.
\end{align*}
We conclude that $h(u)\adj =h(u\adj)$ and $h(u)^{\perp} =h(u^{\perp})$ 
for all $u\in Q$. 

\end{proof}

\begin{credits}
\subsubsection{\ackname} The first author acknowledges support from the Czech Science Foundation (GAČR) project 23-09731L ``Representations of algebraic semantics for substructural logics''. The second author was supported by the Austrian Science Fund (FWF) [10.55776/PIN5424624] and the Czech Science Foundation (GAČR) project 25-20013L ``Orthogonality and Symmetry''. The third author acknowledges support from the Masaryk University project MUNI/A/1457/2023.

\end{credits}
%
%
%
%

\end{document}